\documentclass[11pt,reqno]{amsart}
\usepackage{amsfonts,amssymb,amsmath,amsthm}
\usepackage[margin=1in]{geometry}
\usepackage[hidelinks]{hyperref}
\usepackage{enumerate}
\usepackage{microtype}
\usepackage{mathtools}
\usepackage{cite}
\usepackage{color}
\usepackage{tikz-cd}

\def\set#1{\left\{ #1 \right\}}
\def\abs#1{\left| #1 \right|}
\def\p#1{\left( #1 \right)}
\def\Z{\mathbb{Z}}
\def\Q{\mathbb{Q}}
\def\F{\mathbb{F}}
\def\Gal{\operatorname{Gal}}
\def\GL{\operatorname{GL}}
\def\SL{\operatorname{SL}}
\def\PGL{\operatorname{PGL}}
\def\det{\operatorname{det}}
\def\tr{\operatorname{tr}}
\def\im{\operatorname{im}}
\def\Frob{\operatorname{Frob}}

\def\Mat{\operatorname{Mat}}
\def\rad{\operatorname{rad}}
\def\res{\operatorname{res}}
\def\kronecker#1#2{\p{\frac{#1}{#2}}}

\theoremstyle{plain}
\newtheorem{theorem}{Theorem}
\newtheorem{corollary}[theorem]{Corollary}
\newtheorem{lemma}[theorem]{Lemma}
\newtheorem{proposition}[theorem]{Proposition}

\theoremstyle{definition}

\theoremstyle{remark}
\newtheorem{remark}[theorem]{Remark}

\title{On the effective version of Serre's open image theorem}
\author[Mayle]{Jacob Mayle}
\address{Jacob Mayle, Department of Mathematics, Wake Forest University, Winston-Salem, NC 27109}
\email{maylej@wfu.edu}
\author[Wang]{Tian Wang}
\address{Tian Wang, Max Planck Institute for Mathematics, Vivatsgasse 7, 53111 Bonn, Germany}
\email{twang@mpim-bonn.mpg.de}
\date{\today}

\begin{document}

\begin{abstract} Let $E/\Q$ be an elliptic curve without complex multiplication. By Serre's open image theorem, the mod $\ell$ Galois representation $\overline{\rho}_{E, \ell}$ of $E$ is surjective for each prime number $\ell$ that is  sufficiently large. Under the generalized Riemann hypothesis, we give an explicit upper bound on the largest  prime $\ell$, linear in the logarithm of the conductor of $E$,  such that $\overline{\rho}_{E, \ell}$ is nonsurjective.
\end{abstract}

\maketitle

\section{Introduction}

Let $E$ be an elliptic curve defined over $\Q$. For a prime number $\ell$, let $E[\ell]$ denote the $\ell$-torsion subgroup of $E(\overline{\Q})$ and let $T_\ell(E)$ denote the $\ell$-adic Tate module of $E$. Recall that $E[\ell]$ and $T_\ell(E)$ are free modules of rank two over $\F_\ell$ and $\Z_\ell$, respectively, where $\F_\ell$ denotes the finite field with $\ell$ elements and $\Z_\ell$ denotes the ring of $\ell$-adic integers. Fixing bases, we obtain the module isomorphisms
\[ E[\ell] \cong \F_\ell \oplus \F_\ell \quad \text{and} \quad T_\ell(E) \cong \Z_\ell \oplus \Z_\ell. \]
The absolute Galois group $\Gal(\overline{\Q}/\Q)$ acts coordinate-wise on elements of $E[\ell]$ and on $T_\ell(E)$. These actions respect the above isomorphisms, and give rise to the \textit{mod $\ell$ Galois representation} and \textit{$\ell$-adic Galois representation} of $E$, which are denoted respectively by
\[
\bar{\rho}_{E, \ell}\colon\Gal(\overline{\Q}/\Q) \longrightarrow \GL_2(\F_{\ell}) \quad \text{and} \quad
\rho_{E, \ell}\colon\Gal(\overline{\Q}/\Q) \longrightarrow \GL_2(\Z_{\ell}).
\]

The above Galois representations  carry a considerable amount of information about $E$. For instance, consider the reduction $E_p$ of the curve $E$ at a prime number $p$ that is distinct from $\ell$. The well-known N\'{e}ron--Ogg--Shafarevich criterion gives that $E$ has good reduction at $p$ if and only if $\rho_{E,\ell}$ is unramified at $p$. Further, if $E$ has good reduction at $p$, then $\tr \rho_{E, \ell}(\Frob_p) = a_p(E)$ and $\det \rho_{E, \ell}(\Frob_p) = p$, where $\Frob_p \in \Gal(\overline{\Q}/\Q)$ denotes a Frobenius automorphism associated with $p$ and $a_p(E)$ is defined by the equation $\#E_p(\F_p) = p + 1 - a_p(E)$.

Suppose from now on that $E$ is without complex multiplication, i.e., assume that the geometric endomorphism ring of $E$ is trivial. A celebrated theorem of Serre \cite{Se1972}, known as the open image theorem, establishes that if $\ell$ is sufficiently large, then $\bar{\rho}_{E,\ell}$ is surjective. Let $c(E)$ denote the least positive integer such that if $\ell > c(E)$, then $\bar{\rho}_{E,\ell}$ is surjective.\footnote{The choice to define $c(E)$ in terms of $\bar{\rho}_{E,\ell}$ instead of $\rho_{E,\ell}$ is somewhat arbitrary. Indeed, for a prime number $\ell \geq 5$,  we have that $\bar{\rho}_{E,\ell}$ is surjective if and only if $\rho_{E,\ell}$ is surjective \cite[p. IV-23]{Se1968}.} Serre asked if $c(E) \leq 37$ holds for each elliptic curve $E/\Q$ without complex multiplication. This is known as Serre's uniformity question and, following theoretical advances and vast numerical evidence, is now articulated in precise conjectures of Sutherland \cite[Conjecture 1.1]{Su2016} and Zywina \cite[Conjecture 1.12]{Zy2015a}. Mazur's landmark work on modular curves implies that $c(E) \leq 11$ if $E$ is semistable \cite[Theorem 4]{Ma1978}. Further progress toward resolving Serre's question has since been made by studying modular curves, which we discuss in \S \ref{sub-GL2}. In an adjacent direction, there has been progress in bounding $c(E)$ in terms of invariants of $E$, such as the Faltings height $h_E$ or the conductor $N_E$ of $E$, which we address in this paper.

Serre's original proof of the open image theorem is ineffective and does not give a bound on $c(E)$. The first unconditional bound on $c(E)$ is due to Masser and W\"{u}stholz \cite{MW1993} who proved in 1993 that there exist absolute constants $C_1$ and $\gamma$ such that
\[ c(E) \leq C_1 h_E^\gamma. \]
This bound has subsequently been improved and made explicit in \cite{Pe2001, Lo2015, LeF2016}. Under the now superfluous assumption that $E$ is modular \cite{BC2001}, in 1995 Kraus\cite{Kr1995} bounded $c(E)$ in terms of the conductor of $E$,
\[ c(E) \leq 68 \rad(N_E) (1 + \log \log \rad N_E )^{1/2}, \]
where $\rad n \coloneqq \prod_{p \mid n} p$ denotes the radical of an integer $n$. In 2005, Cojocaru \cite{Co2005} proved independently using a similar approach that 
\[ c(E) \leq \frac{4\sqrt{6}}{3}N_E\prod_{p\mid N_E }\left( 1+\frac{1}{p}\right)^{1/2}. \]
Recently,  Zywina \cite[Proposition 1.8 and Theorem 1.10]{Zy2015b} improved the bounds of Kraus and Cojocaru.

Considerably better bounds for $c(E)$ are known under the assumption of the generalized Riemann hypothesis (GRH) for Dedekind zeta functions.  Assuming GRH, Serre gave an elegant proof in 1981 \cite[Th\'{e}or\`{e}me 22]{Se1981} that there exists an absolute, computable constant $C_2$ such that
\begin{equation} c(E)\leq C_2(\log \rad N_E)(\log\log \rad 2N_E)^3. \label{SeBd} \end{equation}
Serre mentioned without proof in \cite[Note 632.6]{Se1986} that the ``$\log\log$'' term in (\ref{SeBd}) may be removed by employing an $\ell$-adic technique of Faltings \cite[\S 6, pp. 362-363]{Fa1983}. The technique and its extension is called the Faltings-Serre method (see \cite[Section 2]{MR3981316})  and is widely recognized for its important role in establishing various modularity results: such as the modularity of elliptic curves over (infinitely) many imaginary quadratic fields \cite{MR2600560, MR1253207, MR1213108, MR3946721, caraiani2023modularity}, the paramodularity of certain abelian surfaces with a trivial endomorphism ring \cite{MR3981316, MR4349242}, and the modularity of some Calabi-Yau threefolds \cite{MR3007681, MR4057530}. 

In 2014, Larson and Vaintrob \cite[Theorem 1]{LV2013} proved, without applying the Falting-Serre method, that under GRH, there exists an absolute constant $C_3$ such that
\begin{equation} c(E) \leq C_3 \log N_E. \label{LarsonVaintrob} \end{equation}
Not only is their bound linear in $\log N_E$, but it also holds over an arbitrary number field $K$ (with $C_3$ depending only on $K$). However, even over $\Q$, the bound \eqref{LarsonVaintrob} is ineffective in the sense that no method is presently known for computing $C_3$. In order to compute $C_3$ via the proof in \cite{LV2013}, one would need to understand the rational points on the modular curve $X_{ns}^+(\ell)$ for some prime $\ell \geq 53$.

The main result of this paper is an explicit conditional bound on $c(E)$ of the same asymptotic quality as (\ref{LarsonVaintrob}) for elliptic curves over $\Q$. Specifically, we shall prove the following.
\begin{theorem} \label{main-thrm} Assume GRH. If $E/\Q$ is an elliptic curve without complex multiplication, then
\[
c(E) \leq 964 \log \rad(2N_{E}) + 5760,
\]
where $\rad n \coloneqq \prod_{p \mid n} p$ denotes the radical of an integer $n$.
\end{theorem}

Our proof of Theorem \ref{main-thrm} follows the strategy set forth by Serre in \cite[Th\'{e}or\`{e}me 22]{Se1981} and \cite[Note 632.6]{Se1986}. The key improvement comes from the following sharpening of his conditional bound  in the effective version of Faltings's isogeny theorem for elliptic curves \cite[Th\'{e}or\`{e}me 21]{Se1981}.

\begin{theorem} \label{thrm-qisog} Assume GRH. Let $E_1/\Q$ and $E_2/\Q$ be elliptic curves without complex multiplication. Suppose that $E_1$ and $E_2$ are not $\Q$-isogenous. Then there exists a prime number $p$ of good reduction for $E_1$ and $E_2$ such that $a_p(E_1) \neq a_p(E_2)$ satisfying the inequality
\[p\leq (482 \log \rad(2N_{E_1}N_{E_2}) +2880)^2.
\]
\end{theorem}

The structure of our paper is as follows. In \S\ref{cdt-sec}, we give a variant of the effective Chebotarev density theorem due to Bach and Sorenson. We use this tool, together with a refinement of a  technique of Faltings, to prove Theorem \ref{thrm-qisog} in \S\ref{faltings-sec}. With this result in hand, in \S\ref{main-sec}, we follow in the footsteps of Serre's elegant proof of (\ref{SeBd}) to complete our proof of Theorem \ref{main-thrm}. Afterward, we illustrate our result with a numerical example in \S\ref{ex-sec}.

We conclude the introduction with some remarks on extensions of the effective Serre's open image theorem. 

We recall from the modularity theorem  \cite{Wi1995, TaWi1995} that Galois representations of elliptic curves over $\Q$ arise from Galois representations of weight 2 cuspidal eigenforms. Therefore, Theorem \ref{main-thrm} can also be interpreted as an effective open image theorem for modular forms of weight 2. We refer the reader to \cite{BiDi2014, Pe2022} for other effective results for higher weight modular forms.

Let $K$ be a number field and $E/K$ be a non-CM elliptic curve. Serre's open image theorem also applies over a number field, so we can define $c(E)$ in a similar way as before. It is known that there is a uniform bound of $c(E)$ for certain families of $\Q$-curves $E$ over a quadratic field $K$ (see \cite{LeF2016, Le2019}).  

 The main result of this paper is to make explicit the conditional bound of Larson and Vaintrob \cite[Theorem 1]{LV2013} for elliptic curves over $\Q$. However, as their bound holds for elliptic curves over number fields, 
it is natural to ask if we could extend our result to $E/K$. The approach that we follow relies on  Mazur's cyclic isogeny theorem. Generalizing Mazur's result to elliptic curves over arbitrary number fields appears to be challenging. Nonetheless,  assuming GRH, if $K$ is among a certain finite set of quadratic fields $K$,  Banwait \cite{Ba2022} and  Banwait, Najman, and Padurariu \cite{BaNaPa2022},  building on the earlier work of  David \cite{Da2011}, Larson and Vaintrob \cite{LaVa2014}, and  Momose\cite{Mo1995}, proved an analog of Mazur's cyclic isogeny theorem for $E/K$. 
Thus it is promising that one may be able to extend our work to give an explicit open image theorem for elliptic curves defined over these quadratic fields.

\section{Progress toward Serre's uniformity question} \label{sub-GL2}

The most significant progress toward a resolution of Serre's uniformity question comes from studying rational points on certain modular curves. By doing so, one limits the  possibilities for $G_E(\ell) \coloneqq \im \bar{\rho}_{E,\ell}$, i.e., the image of $\bar{\rho}_{E,\ell}$. As $G_E(\ell)$ is a subgroup of $\GL_2(\F_\ell)$, in order to describe what is known in this direction, we first state a well-known classification of subgroups of $\GL_2(\F_\ell)$ (which dates back to Dickson \cite{Di1901}) and give some necessary terminology.

Let $\ell$ be an odd prime number and fix a nonsquare element $\varepsilon \in \F_\ell^\times$. The \textit{split Cartan subgroup} and \textit{non-split Cartan subgroup} of $\GL_2(\F_\ell)$ are, respectively,
\[
\mathcal{C}_s(\ell) \coloneqq \set{\begin{pmatrix} a & 0 \\ 0 & d \end{pmatrix} : a,d \in \F_\ell^\times }
\quad \text{and} \quad
\mathcal{C}_{ns}(\ell) \coloneqq \set{\begin{pmatrix} a & \varepsilon c \\ c & a\end{pmatrix} : a,c \in \F_\ell \text{ and } (a,c) \neq (0,0)}.
\]
Let $\mathcal{C}_s^+(\ell)$ and $\mathcal{C}_{ns}^+(\ell)$ denote the normalizer of $\mathcal{C}_s(\ell)$ and $\mathcal{C}_{ns}(\ell)$, respectively. One may show that 
\begin{equation} \label{NormCartan}
\mathcal{C}^+_{s}(\ell) = \mathcal{C}_{s}(\ell) \cup \begin{pmatrix} 0 & 1 \\ 1 & 0 \end{pmatrix} \mathcal{C}_{s}(\ell)
\quad \text{and} \quad
\mathcal{C}^+_{ns}(\ell) = \mathcal{C}_{ns}(\ell) \cup \begin{pmatrix} 1 & 0 \\ 0 & -1 \end{pmatrix} \mathcal{C}_{ns}(\ell).
\end{equation}
Let $\mathcal{B}(\ell)$ denote the \textit{Borel subgroup} of $\GL_2(\F_\ell)$, i.e., the subgroup of upper triangular matrices,
\[
\mathcal{B}(\ell) \coloneqq \set{\begin{pmatrix} a & b \\ 0 & d \end{pmatrix} : a,d \in \F_\ell^\times \text{ and } b \in \F_\ell}.
\]
Let $A_n$ and $S_n$ denote the alternating and symmetric groups, respectively, on $n$ elements. Finally, for a subgroup $G$ of $\GL_2(\F_\ell)$, let $\overline{G}$ denote the image of $G$ in the projective linear group $\PGL_2(\F_\ell)$. With notation set, we now state the classification. For further details, we refer the reader to \cite[\S2]{Se1972}.
\begin{proposition} \label{Dickson} Let $\ell$ be an odd prime. If $G \subseteq \GL_2(\F_\ell)$ is a subgroup, then
\begin{enumerate}
    \item \label{SL2} $G$ contains $\SL_2(\F_\ell)$,
    \item \label{Borel} $G$ is conjugate to a subgroup of $\mathcal{B}(\ell)$,
    \item \label{Cartan} $G$ is conjugate to a subgroup of $\mathcal{C}_{ns}(\ell)$,
    \item \label{Ns} $G$ is conjugate to a subgroup of $\mathcal{C}_s^+(\ell)$ but not to any subgroup of $\mathcal{C}_s(\ell)$,
    \item \label{Nns} $G$ is conjugate to a subgroup of $\mathcal{C}_{ns}^+(\ell)$ but not to any subgroup of $\mathcal{C}_{ns}(\ell)$, or
    \item \label{Exp} $\overline{G}$ is isomorphic to $A_4$, $S_4$, or $A_5$.
\end{enumerate}
\end{proposition}

Returning to the world of elliptic curves, recall that by the Weil pairing on $E$, the composition
\[ \det \circ \bar{\rho}_{E,\ell}\colon \Gal(\overline{\Q}/\Q) \longrightarrow \F_\ell^\times \]
is the mod $\ell$ cyclotomic character and hence is surjective  \cite[III.8]{Si2009}. In particular, if $\SL_2(\F_\ell) \subseteq G_E(\ell)$, then in fact $G_E(\ell) = \GL_2(\F_\ell)$. Consequently, in order to prove that $\bar{\rho}_{E,\ell}$ is surjective for a particular prime number $\ell$, it suffices to rule out possibilities (\ref{Borel}) through (\ref{Exp}) of Proposition \ref{Dickson} for the group $G_E(\ell)$. For a non-CM elliptic curve, many cases are already ruled out for sufficiently large $\ell$. Indeed, Serre ruled out (\ref{Cartan})  for $\ell > 2$ \cite[\S 5.2]{Se1972} and (\ref{Exp}) for $\ell > 13$ \cite[Lemme 18]{Se1981}; Mazur ruled out (\ref{Borel}) for $\ell > 37$ \cite[Theorem 3]{Ma1978}; Bilu, Parent, and Rebolledo ruled out (\ref{Ns}) for $\ell > 7$ and $\ell\neq 13$ \cite[Corollary 1.2]{BP2013}; Balakrishnan, Dogra, Netan, M\"{u}ller, Tuitman, and Vonk ruled out  (\ref{Ns}) for $\ell=13$ \cite[Theorem 1.2]{MR3961086}. Therefore, all but (\ref{Nns}) are ruled out for each prime number $\ell > 37$, as recorded below.

\begin{theorem} \label{SMBPR} Let $E/\Q$ be an elliptic curve without complex multiplication. If $\ell$ is a prime number such that $\ell > 37$, then either $\bar{\rho}_{E,\ell}$ is surjective or $G_E(\ell)$ is conjugate to a subgroup of $\mathcal{C}_{ns}^+ (\ell)$ but is not conjugate to any subgroups of $\mathcal{C}_{ns}(\ell)$.
\end{theorem}

\section{The effective Chebotarev density theorem} \label{cdt-sec}

In this section, we offer a modest extension of the version of the effective Chebotarev density theorem  given in \cite{BS1996}. This extension will serve as a crucial tool in our proof of Theorem \ref{main-thrm}. First, let us fix some relevant notation.

Let $K$ be a number field with absolute discriminant $d_K$ and ring of integers $\mathcal{O}_K$. For a prime ideal $\mathfrak{p} \subseteq \mathcal{O}_K$, we write $\mathfrak{p} \mid p$ to indicate that $\mathfrak{p}$ lies above a prime number $p$. Let $v_{\mathfrak{p}}\colon K \to \Z$ denote the normalized $\mathfrak{p}$-adic valuation on $K$. We write $N(\mathfrak{p})$ for the ideal norm of $\mathfrak{p}$, which extends multiplicatively to arbitrary ideals of $\mathcal{O}_K$. One has that  
$ v_\mathfrak{p}(p)=e_{\mathfrak{p}}$  and  $N(\mathfrak{p})=p^{f_{\mathfrak{p}}}$,
where $e_{\mathfrak{p}}$ and $f_{\mathfrak{p}}$ are the ramification index and inertia degree of $\mathfrak{p}$, respectively.

Assume that $K/\Q$ is Galois. For a prime number $p$ that is unramified in $K/\Q$, write $(\tfrac{p}{K/\Q})$ to denote the Artin symbol of $K/\Q$ at $p$. Let $C$ be a subset of the Galois group $\Gal(K/\Q)$ that is closed under conjugation. Associated with $C$, consider the counting function
\[ \pi_{C}(x) \coloneqq \# \set{p \leq x : p \text{ is unramified in $K/\Q$ and } \! \p{\frac{p}{K/\Q}} \subseteq C}. \]
The Chebotarev density theorem states that
\begin{equation} \label{cdt}
    \pi_{C}(x) \sim \frac{\#{C}}{\#{\Gal(K/\Q)}} \pi(x),
\end{equation} 
where $\pi(x)$ is the prime counting function and $f \sim g $ means that $\lim_{x \to \infty} \frac{f(x)}{g(x)} = 1$. 

Assuming GRH, Lagarias and Odlyzko \cite{LO1975} gave an effective version of the Chebotarev density theorem that provides an error term in (\ref{cdt}).  Moreover, they showed that there exists an absolute constant $k$ such that the least prime number $p$ with
\begin{equation} \p{\frac{p}{K/\Q}} \subseteq C \label{artin-in-C} \end{equation}
satisfies the inequality $p \leq k (\log d_K)^2$. Oesterl\'{e} \cite{Oe1979} stated that the absolute constant $k$ may be taken to be $70$. Subsequently, Bach and Sorenson offered the following improvement.

\begin{theorem} \label{effective-cdt} Assume GRH. Let $K$ be a Galois number field and let $C \subseteq \Gal(K/\Q)$ be a nonempty subset that is closed under conjugation. Then there exists a  prime number $p$ that is unramified in $K/\Q$ for which (\ref{artin-in-C}) holds that satisfies the inequality
\[ p \leq  (a \log d_K + b [K:\Q] + c)^2, \]
where $a,b,$ and $c$ are absolute constants that may be taken to be $4, 2.5,$ and $5$, respectively, or may be taken to be the improved values given in \cite[Table 3]{BS1996} associated with $K$.
\end{theorem} 
\begin{proof} See \cite[Theorem 5.1]{BS1996}. \end{proof}

For our application, we need an extension of Theorem \ref{effective-cdt} that allows for the avoidance of a prescribed set of primes. 

\begin{corollary}\label{ECDT-cor} Assume GRH. Let $K$ be a Galois number field, let $m$ be a squarefree positive integer, and set $\tilde{K} \coloneqq K(\sqrt{m})$. Let $C \subseteq \Gal(K/\Q)$ be a nonempty subset that is closed under conjugation. Then there exists a prime number $p$ not dividing $m$ that is unramified in $K/\Q$  for which (\ref{artin-in-C}) holds that satisfies  the inequality
\begin{equation} \label{ECDT-cor-ineq}
    p \leq  (\tilde{a} \log d_{\tilde{K}} + \tilde{b} [\tilde{K}:\Q] + \tilde{c})^2,
\end{equation} 
where  $\tilde{a},\tilde{b},\tilde{c}$ are absolute constants that may be taken to be $4, 2.5,$ and $5$, respectively, or may be taken to be the improved values given in \cite[Table 3]{BS1996} associated with $\tilde{K}$.
\end{corollary}
\begin{proof}
Notice that $\tilde{K}$ is Galois over $\Q$  since both $K$ and $\Q(\sqrt{m})$ are Galois over $\Q$. If $\tilde{K} = K$, then each prime number dividing $m$ is ramified in $K/\Q$, so Theorem \ref{effective-cdt} provides the desired result. Thus, we assume that $\tilde{K} \neq K$. Then $ K \cap \Q(\sqrt{m}) = \Q $, so by \cite[Theorem VI 1.14]{La2005}, we have that
\[ \Gal(\tilde{K}/\Q) \cong \Gal(K/\Q) \times (\Z/2\Z). \]

Let $\res\colon \Gal(\tilde{K}/\Q) \longrightarrow \Gal(K/\Q)$ denote the restriction map and consider the subset $\tilde{C} \coloneqq \res^{-1}(C)$ of $\Gal(\tilde{K}/\Q)$. Since $C$ is closed under conjugation in $\Gal(K/\Q)$, $\tilde{C}$ is closed under conjugation in $\Gal(\tilde{K}/\Q)$. By applying Theorem \ref{effective-cdt} to $\tilde{K}$ and $\tilde{C}$, we obtain a prime number $p$ that is unramified in $\tilde{K}$ for which (\ref{artin-in-C}) holds and satisfies the inequality (\ref{ECDT-cor-ineq}). Note that $\tilde{K}$ is ramified at the ramified primes of $K$ and at the prime divisors of $m$ (and possibly at 2). Thus $p$ is unramified in $K$ and does not divide $m$. Finally, because $\p{\tfrac{p}{\tilde{K}/\Q}} \subseteq \tilde{C}$ and $\res \p{\p{\tfrac{p}{\tilde{K}/\Q}}} = \p{\tfrac{p}{K/\Q}}$, we have that  $\p{\tfrac{p}{{K}/\Q}} \subseteq {C}$.
\end{proof}

In the corollary, we see that $p$ is bounded above in terms of $[\tilde{K}:\Q]$ and $\log d_{\tilde{K}}$. In our application, the degree $[\tilde{K}:\Q]$ will be absolutely bounded. Thus, it will remain to bound  $\log d_{\tilde{K}}$. For this purpose, we employ the following lemma, which can be found in \cite[Proposition 6]{Se1981}.

\begin{lemma} \label{lem-logdK} If $K/\Q$ is a nontrivial finite Galois extension, then
\[
(\tfrac{1}{2}\log 3) [K:\Q]\leq \log d_K \leq \p{[K:\Q] - 1} \log \rad(d_K) + [K:\Q] \log([K:\Q]).
\]
\end{lemma}
\begin{proof} The left-hand inequality follows from the Minkowski bound for the discriminant \cite[p.\ 139]{Se1981}. For the right-hand inequality, let $\mathfrak{D}_K \subseteq \mathcal{O}_K$ denote the different ideal of $K$ and note that
\[ d_K = N(\mathfrak{D}_K) = \prod_{p \mid d_K} p^{v_p(N(\mathfrak{D}_K))}. \]
By taking logarithms, we obtain
\begin{equation} \label{log-dK}
\log d_K = \sum_{p \mid d_K} v_p(N(\mathfrak{D}_K)) \log p = \sum_{p \mid d_K} \sum_{\mathfrak{p} \mid p} f_{\mathfrak{p}} v_{\mathfrak{p}}(\mathfrak{D}_K) \log p.
\end{equation}
For each prime ideal $\mathfrak{p} \subseteq \mathcal{O}_K$ lying above $p$, we have that
\[
    v_{\mathfrak{p}}(\mathfrak{D}_K)=e_{\mathfrak{p}}-1+s_{\mathfrak{p}}
\]
for some integer $s_{\mathfrak{p}}$ satisfying  $0\leq s_{\mathfrak{p}} \leq v_{\mathfrak{p}}(e_{\mathfrak{p}})$ (see, e.g., \cite[Theorem 2.6, p. 199]{Ne1999}). Thus 
\begin{equation} \label{sum-1}
\sum_{\mathfrak{p}\mid p} f_{\mathfrak{p}} v_{\mathfrak{p}}(\mathfrak{D}_K)
=\sum_{\mathfrak{p}\mid p} f_{\mathfrak{p}}(e_{\mathfrak{p}}-1)+
\sum_{\mathfrak{p}\mid p}f_{\mathfrak{p}}s_{\mathfrak{p}} 
\leq [K:\Q] - 1 + \sum_{\mathfrak{p} \mid p} f_{\mathfrak{p}} v_{\mathfrak{p}}(e_{\mathfrak{p}}).
\end{equation}
Since $K/\Q$ is Galois, $e_{\mathfrak{p}}$ divides $[K:\Q]$. Thus
$v_p(e_{\mathfrak{p}}) \leq v_p([K:\Q])$. Hence 
\begin{equation} \label{sum-2}
\sum_{\mathfrak{p} \mid p} f_{\mathfrak{p}} v_{\mathfrak{p}}(e_{\mathfrak{p}}) = \sum_{\mathfrak{p} \mid p} f_{\mathfrak{p}} e_{\mathfrak{p}} v_p(e_{\mathfrak{p}}) \leq v_p([K:\Q]) \sum_{\mathfrak{p} \mid p}  f_{\mathfrak{p}} e_{\mathfrak{p}} = v_p([K:\Q]) [K:\Q].
\end{equation}
Observe that
\begin{equation}\label{sum-3}
   \sum_{p \mid d_K}v_p([K:\Q]) [K:\Q]\log p\leq [K:\Q] \log [K:\Q].
\end{equation}

Applying  (\ref{sum-1}),  (\ref{sum-2}) and  (\ref{sum-3}) to (\ref{log-dK}), we obtain
\[
\log d_K \leq  ([K:\Q] - 1) \sum_{p \mid d_K} \log p + [K:\Q] \log [K:\Q].
\]
The claimed inequality now follows by noting that $\sum_{p \mid d_K} \log p = \log \rad (d_K)$.
\end{proof}

\section{An effective isogeny theorem for elliptic curves} \label{faltings-sec}

The objective of this section is to provide an improved conditional bound on the effective version of Faltings's isogeny theorem for elliptic curves. We begin in \S\ref{NOS-sec} with some preliminaries.

\subsection{Ramified primes} \label{NOS-sec}

Let $A/\Q$ be an abelian variety. For a positive integer $m$, let $\Q(A[m])$ denote the $m$-division field of $A$, i.e., the field obtained by adjoining to $\Q$ the coordinates of all points of $A(\overline{\Q})$ of order dividing $m$. For a prime number $\ell$, we write
\[ \Q(A[\ell^\infty]) \coloneqq \bigcup_{k = 1}^\infty \Q(A[\ell^k]). \]

We now recall Serre and Tate's extension of the criterion of N\'{e}ron--Ogg--Shafarevich to abelian varieties in order to specify which primes ramify in the infinite degree algebraic extension $\Q(A[\ell^\infty])/\Q$. 

\begin{theorem}\label{NOS} Let $A/\Q$ be an abelian variety. For a prime number $p$, the following are equivalent:
\begin{enumerate}
    \item\label{NOS-1} $A$ has good reduction at $p$,
    \item\label{NOS-2} $\Q(A[m])/\Q$ is unramified at $p$ for each positive integer $m$, not divisible by $p$, and
    \item\label{NOS-3} $\Q(A[m])/\Q$ is unramified at $p$ for infinitely many positive integers $m$, not divisible by $p$.
\end{enumerate}
\end{theorem}
\begin{proof} See \cite[Theorem 1]{ST1968}.
\end{proof}

\begin{corollary} \label{cor-NOS} Let $B_A$ be the product of $\ell$ and the primes of bad reduction for $A$. The extension $\Q(A[\ell^\infty])/\Q$ is ramified at exactly the prime divisors of $B_A$.
\end{corollary}
\begin{proof} First, we note that $\ell$ is ramified in $\Q(A[\ell^\infty])/\Q$. Indeed, let $\zeta_{\ell^2} \in \overline{\Q}$ be a primitive  $\ell^2$-root of unity. It follows from the Weil pairing on $E[\ell^2]$ (see the exercise in \cite[p. 55]{Se1997}) that
\[
\Q \subseteq
\Q(\zeta_{\ell^2}) \subseteq
\Q(A[\ell^2])  \subseteq
\Q(A[\ell^\infty]). \]
Because $\ell$ ramifies in $\Q(\zeta_{\ell^2})/\Q$, it follows that $\ell$  ramifies in $\Q(A[\ell^\infty])/\Q$.\footnote{If $\ell \neq 2$, then it suffices to consider $\Q \subseteq
\Q(\zeta_{\ell}) \subseteq \Q(A[\ell])  \subseteq \Q(A[\ell^\infty])$ to observe that $\ell$ ramifies in $\Q(A[\ell^\infty])/\Q$.}

We continue the proof by showing that the extension $\Q(A[\ell^\infty])/\Q$ is unramified at exactly the primes not dividing $B_A$. Let $p$ be a prime number such that $p\nmid B_A$. Then $A$ has good reduction at $p$, so it follows from the equivalence of (\ref{NOS-1}) and (\ref{NOS-2}) in Theorem \ref{NOS} that in the chain of subfields
\[\Q\subseteq \Q(A[\ell])\subseteq \Q(A[\ell^2])\subseteq\cdots \subseteq \Q(A[\ell^\infty]),\] 
the prime $p$ is unramified in  $\Q(A[\ell^n])/\Q$ for each $n\geq 0$. Thus  $p$ is unramified in $\Q(A[\ell^\infty])/\Q$. Now suppose that $p$ is distinct from $\ell$ and is unramified in $\Q(A[\ell^\infty])/\Q$. Then in fact $p$ is unramified in $\Q(A[\ell^n])/\Q$ for each $n\geq 0$. By the equivalence of (\ref{NOS-1}) and (\ref{NOS-3}) in Theorem \ref{NOS}, $p$ is a prime of good reduction of $A$, so $p \nmid B_A$. 
\end{proof}

\begin{remark}\label{NOS-gal}
Equivalently, $\rho_{A, \ell}$ is ramified at exactly the prime divisors of $B_A$.
\end{remark}

\begin{remark} \label{prod} Let $A = E_1 \times E_2$ be the product of two elliptic curves $E_1$ and $E_2$, with conductors $N_{E_1}$ and $N_{E_2}$, respectively. It follows from  Theorem \ref{NOS} that $A$ has good reduction at a prime $\ell$ if and only if $E_1$ and $E_2$ both have good reduction at $\ell$. Thus $\rad(B_A) = \rad(\ell N_{E_1} N_{E_2})$.
\end{remark}

\subsection{Effective isogeny theorem} In 1968, Serre proved the isogeny theorem for elliptic curves with non-integral $j$-invariant \cite[p. IV-14]{Se1968}. This was  generalized by Faltings for arbitrary abelian varieties \cite[Korollar 2, p. 361]{Fa1983}. For elliptic curves $E_1/\Q$ and $E_2/\Q$, the isogeny theorem gives that $E_1$ and $E_2$ are $\Q$-isogenous if and only if for each prime $p \nmid N_{E_1} N_{E_2}$, one has that
\begin{equation} \label{isog-ap}  a_p(E_1) = a_p(E_2). \end{equation}
Provided that $E_1$ and $E_2$ are without complex multiplication and not $\Q$-isogenous, Serre proved in \cite[Th\'{e}or\`{e}me 21]{Se1981} that under GRH, the least prime number $p \nmid N_{E_1} N_{E_2}$ for which equality in  (\ref{isog-ap}) fails to hold satisfies the inequality
\begin{equation} \label{serre-isog} p \leq C_4 (\log \rad(N_{E_1} N_{E_2}) )^2 (\log \log \rad(2N_{E_1}N_{E_2}))^{12}, \end{equation}
for some absolute, computable constant $C_4$. In this section, we improve upon the bound in (\ref{serre-isog}) by  removing the ``$\log\log$'' factor.

We begin by offering a refinement of a proposition due to Serre. Serre mentioned his result, without proof, in \cite[Note 632.6]{Se1986}. Much later, he communicated an elegant proof that appears in \cite[Theorem 4.7]{BK2016}. Our proof builds on the one appearing in \cite{BK2016} by considering a quotient by scalar matrices as in the proof of \cite[Th\'{e}or\`{e}me 21']{Se1981}. Our bound on the order of $G$ (as below) coincides with Serre's when $\ell = 2$, which is the prime for which we will apply the result in the 
proof of Theorem \ref{thrm-qisog}. However, it does offer an improvement in the case when $\ell \neq 2$, so we find it worthwhile to include it nonetheless.

\begin{proposition} \label{prop-serre}  
Let $\ell$ be a prime number and $r$ be a positive integer. Let $\Gamma$ be a group and $\rho_1,\rho_2\colon \Gamma \longrightarrow \GL_r(\Z_\ell)$ be group homomorphisms. Suppose  there is an element $\gamma \in \Gamma$ such that $\tr \rho_1(\gamma) \neq \tr \rho_2(\gamma)$. Then there exists a quotient $G$ of $\Gamma$ and a subset $C \subseteq G$ for which
\begin{enumerate}
\item \label{G-ord} the order of $G$ is at most $\frac{\ell^{2r^2}-1}{\ell-1}$, 
\item the set $C$ is nonempty and closed under conjugation in $G$, and
\item if the image in $G$ of an element $\gamma \in \Gamma$ belongs to $C$, then $\tr \rho_1(\gamma) \neq \tr \rho_2(\gamma)$.
\end{enumerate}
\end{proposition}
\begin{proof} Let $M\coloneqq \Mat_{r \times r}(\Z_\ell)$ be the $\Z_\ell$-algebra of the $r \times r$ matrices with coefficients in $\Z_\ell$. 
Let $A$ denote the $\Z_\ell$-algebra  generated by the image of $\Gamma$ under the product map
\[ 
\rho_1 \times \rho_2 \colon \Gamma \longrightarrow  \GL_r(\Z_\ell) \times \GL_r(\Z_\ell). 
\]
Let $G'$ be the image of $\Gamma$ under $\rho_1 \times \rho_2$ in $A/\ell A$. 
Because of the existence of the identity element in $\Gamma$, the algebra $A$ contains the set of scalar matrices \[\Lambda_{2r} \coloneqq \set{(\mu I_r, \mu I_r) : \mu \in \Z_\ell}.\] We write $H$ to denote the image of 
$\Lambda_{2r}\backslash \ell\Lambda_{2r}$ in $A/\ell A$. Then we have the group  isomorphism $H\cong \Z/\ell \Z^{\times}$, as the image of $(\mu I_r, \mu I_r)$ in $A/\ell A$ is determined by the image of $\mu$ in $\Z/\ell \Z$.  Clearly, $H\cap G'$ is a normal subgroup of $G'$. 
Consider the group $G \coloneqq  G'/ (H\cap G')$. From the second isomorphism theorem in group theory, we have  
$
G \cong G'H/H.
$
Since the rank of $A$ as a free $\Z_{\ell}$-module is at most $2r^2$ and both $G'H$ and $H$ are groups in $A/\ell A$, we obtain the bound  
\[
|G|= |G'H/H| \leq \frac{\ell^{2r^2}-1}{\ell-1}.
\]

Let $m$ be the largest nonnegative integer such that for each $\gamma\in \Gamma$, one has that
\[
\tr \rho_1(\gamma)\equiv \tr \rho_2(\gamma)  \pmod{\ell^m}.
\]
As $A$ is a $\Z_{\ell}$-algebra generated by the image of $\Gamma$ under $\rho_1 \times \rho_2$,  it follows that the congruence $\tr x_1 \equiv \tr x_2 \pmod{\ell^m}$ holds for each $(x_1,x_2) \in A$.
We obtain  the $\Z_\ell$-module homomorphism $\lambda\colon A \longrightarrow \Z_\ell$ given by
\[
\lambda(x_1,x_2) = \ell^{-m}(\tr x_1 - \tr x_2). 
\]
Since $\lambda(\ell A) \subseteq \ell\Z_\ell$, we may consider the induced $\Z/\ell\Z$-module homomorphism $\bar{\lambda}\colon A/\ell A \longrightarrow \Z/\ell\Z$. 

Let $C$ be the set of elements in $G$ whose preimages in $G'$ all take nonzero values under $\bar{\lambda}$. From the definition of $m$ and the linearity of the trace map, there exists a $\gamma_0 \in \Gamma$ such that
\[ \tr \mu \rho_1(\gamma_0) \not\equiv \tr \mu \rho_2(\gamma_0) \pmod{\ell^{m+1}} \quad \forall \mu \in \Z_\ell^\times. \]

Note that the image of $(\rho_1 \times \rho_2)(\gamma_0)$ in $G$ is contained in $C$, so $C$ is nonempty. Further, $C$ is closed under conjugation because trace is invariant under conjugation. Finally, suppose that $\gamma \in \Gamma$ is such that the image of $\gamma$ in $G$ is contained in $C$. Then $ \lambda((\rho_1 \times \rho_2)(\gamma)) \not\in \ell \Z_\ell$, and in particular $\tr \rho_1(\gamma) \neq \tr \rho_2(\gamma)$.
\end{proof}

We now employ the above proposition, together with the effective Chebotarev density theorem in the form of Corollary \ref{ECDT-cor} to give an improvement on the bound in (\ref{serre-isog}), as  recorded in Theorem  \ref{thrm-qisog}.

\begin{proof}[
Proof of Theorem \ref{thrm-qisog}]
Let $A \coloneqq E_1 \times E_2$ and apply Proposition \ref{prop-serre} with $\ell \coloneqq 2$, $r \coloneqq 2$,  $\Gamma \coloneqq \Gal(\Q(A[2^\infty])/\Q)$, and $\rho_i \coloneqq \rho_{E_i,2}$ for each $i = 1,2$. Let $G$ and $C$ be as in the conclusion of Proposition \ref{prop-serre}. Let $K$ be a subfield of  $\Q(A[2^\infty])$ for which $\Gal(K/\Q) = G$, which exists by the fundamental theorem of infinite Galois theory. From Proposition \ref{prop-serre}, the size of $G$ is bounded above by $255$. Since $\Q(A[2^\infty]) = \bigcup_k \Q(A[2^k])$, it follows that $K \subseteq \Q(A[2^n])$ for some $n$.
 Thus $[K:\Q]$ divides $[\Q(A[2^n]): \Q ]$, which divides $|\GL_2(\Z/2^n\Z)|^2 = (6 \cdot 16^{n-1})^2$.  One can check that the largest divisor of $(6 \cdot 16^{n-1})^2$ that is at most 255 is 192. Thus $\abs{G} = [K : \Q] \leq 192$. 
 
  Applying Corollary \ref{ECDT-cor} with $K$ and $C$ as above and $m \coloneqq \rad(2N_{E_1}N_{E_2})$, we obtain a prime number $p$ not dividing $m$ such that $(\frac{p}{K/\Q}) \subseteq C$ which satisfies inequality (\ref{ECDT-cor-ineq}). As $\Frob_p\mid_{K} = (\frac{p}{K/\Q})$, it follows from Proposition \ref{prop-serre} that
\[ \tr \rho_{E_1,2}(\Frob_p) \neq \tr \rho_{E_2,2}(\Frob_p). \]
Consequently,  $a_p(E_1) \neq a_p(E_2)$. It remains to show that $p$ satisfies the claimed bound.

As in the statement of Corollary \ref{ECDT-cor}, let $\tilde{K} \coloneqq K(\sqrt{m})$. We have that
\begin{equation*}
[\tilde{K}:\Q] \leq 2[K:\Q] \leq 2 \cdot 192 = 384.
\end{equation*}
Thus by Corollary \ref{ECDT-cor},
\begin{equation*}
    p \leq \p{\tilde{a} \log d_{\tilde{K}} + 384 \tilde{b} + \tilde{c}}^2.
\end{equation*}
where  $\tilde{a},\tilde{b},\tilde{c}$ are absolute constants that may be taken to be $4, 2.5,$ and $5$, respectively, or may be taken to be the improved values given in \cite[Table 3]{BS1996} associated with $\tilde{K}$. 

Thus if $\log d_{\tilde{K}} \leq 100$, then by applying Theorem \ref{effective-cdt} with the constants $4, 2.5,$ and $5$, we find that
\[ p \leq (4 \cdot 100 + 2.5 \cdot 384 + 5)^2 = 1863225. \]
If $100\leq \log d_{\tilde{K}} \leq 1000$, then by Theorem \ref{effective-cdt} with improved constants from \cite[Table 3]{BS1996}, we find that
\[ p \leq (1.755 \cdot 1000 + 0.23 \cdot 384 + 6.8)^2 = 3422944.0144. \]
Next, note that for all real numbers $x \geq 1000$, we have that
\[
1.257 x + 7.3 \geq a'x + b'd + c'
\]
for all $(a',b',c')$ that appears as any entry in the last three rows of \cite[Table 3]{BS1996}, where $d$ is the maximal degree for the corresponding column (and $d = 384$ for the last column). Thus if $\log d_{\tilde{K}} \geq 1000 $, then 
\begin{equation} \label{E:pBd}
    p \leq (1.257 \log d_{\tilde{K}} + 7.3)^2.
\end{equation} 

Therefore, in all cases, we have that
\begin{equation} \label{tc-1}
p \leq \max(3422944.0144, (1.257 \log d_{\tilde{K}} + 7.3)^2).
\end{equation}

We have that by Remark \ref{prod}, $K/\Q$ is unramified outside of the prime divisors of $m = \rad(2N_{E_1}N_{E_2})$.  As $\tilde{K}$ is the compositum of $K$ and $\Q(\sqrt{m})$, the primes that ramify in $\tilde{K}$ are precisely those that ramify in $K$ or in $\Q(\sqrt{m})$. Thus, since $\rad(d_{\Q(\sqrt{m})}) = \rad(2N_{E_1}N_{E_2})$ and $\rad(d_K) \mid \rad(2N_{E_1}N_{E_2})$,
\[ \rad d_{\tilde{K}} = \rad(d_{K} d_{\Q(\sqrt{m})}) = \rad(2N_{E_1}N_{E_2}).\]
Hence by Lemma \ref{lem-logdK},
\begin{equation} \label{tc-4}
\log d_{\tilde{K}} \leq 383 \log \rad(2N_{E_1}N_{E_2}) + 384 \log(384).
\end{equation}
Using the trivial inequality $2N_{E_1}N_{E_2} \geq 2$, we observe that
\begin{equation} \label{tc-3} (1.257( 383 \log \rad(2N_{E_1}N_{E_2}) + 384 \log(384))  + 7.3)^2 \geq 3422944.0144. \end{equation}
Considering (\ref{tc-1}), (\ref{tc-4}), and (\ref{tc-3}), 
we conclude that
\[
p\leq  (1.257( 383 \log \rad(2N_{E_1}N_{E_2}) + 384 \log(384))  + 7.3)^2.
\]
Partially expanding the right-hand side above gives the claimed bound for $p$.
\end{proof}

\begin{remark} \label{R:Table} In the proof above, for fields $\tilde{K}$ with $\log d_{\tilde{K}} \leq 100$, we use the general bound of \cite[Theorem 5.1]{BS1996} rather than the improved bounds appearing in Table 3 of \cite{BS1996}. We do so because Table 3 does not give constants in boxes where some combination of degree and discriminant would violate Minkowski's theorem. This only affects certain entries in the table for which the logarithm of the absolute value of the discriminant is less than $100$. For example, the maximal totally real subfield of $\Q(\zeta_7)$ has degree $3$ and discriminant $49$, yet the table gives no constants for number fields with  degree $3$-$4$ for which the logarithm of the absolute value of the discriminant is less than $5$.
\end{remark}

\section{A bound on the largest non-surjective prime} \label{main-sec}

We begin in \S\ref{quad-sec} with some necessary background on quadratic Galois characters and quadratic twists of elliptic curves. In \S\ref{main-subsec}, we put  together the pieces to complete the proof of Theorem \ref{main-thrm}. Finally, we present a numerical example in \S\ref{ex-sec} that illustrates the theorem.

\subsection{Quadratic twists} \label{quad-sec}
By a \textit{quadratic Galois character}, we mean a surjective group homomorphism
\[ \chi \colon \Gal(\overline{\Q}/\Q) \longrightarrow \set{\pm 1}. \]
Since $\ker \chi$ is an index two subgroup of $\Gal(\overline{\Q}/\Q)$, there exists a nonzero squarefree integer $D$ such that $\ker \chi = \Gal(\overline{\Q}/\Q(\sqrt{D}))$. Consequently $\chi$ factors through $\Gal(\overline{\Q}/\Q) / \ker \chi \cong \Gal(\Q(\sqrt{D})/\Q)$,
\begin{equation*}
\begin{tikzcd}
\Gal(\overline{\Q}/\Q) \ar[dr,swap,"\res_D"] \ar[rr,"\chi"]& & \set{\pm 1} \\
 & \Gal(\Q(\sqrt{D})/\Q) \ar[ur,swap,"\sim"] &
\end{tikzcd}
\end{equation*}
where $\res_D$ denotes the restriction homomorphism. Thus, for each $\sigma \in \Gal(\overline{\Q}/\Q)$,
\begin{equation} \label{chi}
\chi(\sigma) =
\begin{cases}
    1 & \sigma(\sqrt{D}) = \sqrt{D} \\
    - 1 & \sigma(\sqrt{D}) = - \sqrt{D}. \\
\end{cases}
\end{equation}
We write $\chi_D$ to denote the quadratic Galois character described by (\ref{chi}). Note that each quadratic Galois character may be written as $\chi_D$ for a unique squarefree integer $D$.

For a prime number $p$, let $I_p$ and $I_{p}(D)$ denote the inertia subgroups of $\Gal(\overline{\Q}/\Q)$ and $\Gal(\Q(\sqrt{D})/\Q)$, respectively. One says that $\chi_D$ is \textit{unramified} at $p$ if $I_p \subseteq \ker \chi_D$. From the description of $\chi_D$ given in (\ref{chi}) and upon noting that $\res_D(I_p) = I_p(D)$, we see that
\begin{equation} \label{chiD-ram}
\chi_D \text{ is unramified at } p  \iff \Q(\sqrt{D})/\Q \text{ is unramified at } p. 
\end{equation}
If $p \nmid D$, then $\p{\tfrac{p}{\Q(\sqrt{D})/\Q}}(\sqrt{D}) = \kronecker{D}{p} \sqrt{D}$ by \cite[p.\ 88]{Le2013}, where $\kronecker{D}{p}$ denotes the Legendre symbol of $D$ with respect to $p$. Thus, from (\ref{chi}), we have that 
\begin{equation} \label{kron-ap} 
\chi_D\left(\p{\tfrac{p}{\Q(\sqrt{D})/\Q}}\right) = \kronecker{D}{p}. \end{equation}

Now consider an elliptic curve $E/\Q$ given by a Weierstrass equation
\begin{equation}\label{model-EC}
y^2 + a_1 xy + a_3 y = x^3 + a_2 x^2 + a_4 x + a_6. 
\end{equation}
The \textit{quadratic twist} of $E$ by $\chi_D$ (or, equivalently, by $D$) is the elliptic curve $E_D/\Q$ given by
\begin{equation}\label{twist-model-ECs}
y^2 + a_1 xy + a_3 y = x^3 + \p{a_2D + a_1^2 \frac{D-1}{4}} x^2 + \p{a_4 D^2 + a_1 a_3 \frac{D^2-1}{2}} x + a_6 D^3 + a_3^2 \frac{D^3-1}{4}. 
\end{equation}
See \cite[\S 4.3]{Co1999} or \cite[X.2]{Si2009} for background on quadratic twists. By taking (\ref{model-EC}) to be a minimal model for $E$, upon computing and comparing the discriminants of (\ref{model-EC}) and (\ref{twist-model-ECs}), we find that
\begin{equation}\label{twist-cond}
\rad(N_{E_D}) \quad \text{divides} \quad \rad(2D N_E). 
\end{equation}
Further, if $p$ is a prime number such that $p \nmid 2 D N_E$, then by \cite[Exercise 4.10]{Wa2003} and (\ref{kron-ap}),
\begin{equation}\label{ap-chi}
a_p(E) = \chi_D(\Frob_p) a_p(E_D).
\end{equation}

We conclude with a lemma about nontrivial quadratic twists (see, e.g., \cite[p. 199]{Se1981}).

\begin{lemma}\label{non-isogenous-lem} If $D \neq 1$ is squarefree and $E/\Q$ is without complex multiplication, then $E$ and $E_D$ are not $\Q$-isogenous.
\end{lemma}
\begin{proof} We know from the Chebotarev density theorem that the natural density of primes $p$ for which $\chi_D(\Frob_p) = -1$ is $\frac{1}{2}$. For such a $p$, if $p \nmid 2 DN_E$, then $a_p(E) = - a_p(E_D)$ by (\ref{ap-chi}). Thus either $a_p(E) \neq a_p(E_D)$ or $a_p(E) = 0$. The density of primes $p$ for which $p \nmid 2 DN_E$ and $a_p(E) = 0$ is 0 by \cite[p. 123]{Se1981}, \cite[p. 131]{El1991}. Thus there exists a prime $p \nmid 2 D N_E$ such that $a_p(E) \neq a_p(E_D)$. As such, $E$ and $E_D$ are not $\Q$-isogenous.
\end{proof}

\subsection{Completing the proof} \label{main-subsec}

We now turn to the problem of bounding $c(E)$. Suppose that $\ell$ is an odd prime such that $G_E(\ell)$ satisfies (\ref{Nns}) of  Proposition \ref{Dickson}. With an appropriate choice of $\F_\ell$-basis  of $E[\ell]$ in defining $\bar{\rho}_{E,\ell}$, we may assume that $G_E(\ell) \subseteq \mathcal{C}_{ns}^+(\ell)$ and $G_E(\ell) \not\subseteq \mathcal{C}_{ns}(\ell)$. Following Serre, we  consider the quadratic Galois character given by the composition
\begin{equation} \label{epsilon}
\epsilon_\ell \colon \Gal(\overline{\Q}/\Q) \overset{\bar{\rho}_{E,\ell}}{\longrightarrow} G_E(\ell) \longrightarrow \frac{\mathcal{C}_{ns}^+(\ell)}{\mathcal{C}_{ns}(\ell)} \overset{\sim}{\longrightarrow} \set{\pm 1}.
\end{equation}
We list some basic properties of $\epsilon_\ell$, which are previously noted in \cite[p.\ 317]{Se1972} and \cite[p.\ 18]{Co2005}.

\begin{lemma} \label{lemma-epsilon} With the above notation and assumptions, $\epsilon_\ell$ satisfies the following properties. 
\begin{enumerate}
    \item \label{lemma-epsilon-unramified} For each prime $p \nmid N_E$, $\epsilon_\ell$ is unramified at $p$.
    \item \label{lemma-epsilon-cond} One has that  $\epsilon_\ell = \chi_D$ for some integer $D \mid N_E$. 
    \item \label{lemma-epsilon-ap} For each prime $p \nmid N_E$, if $\epsilon_\ell(\Frob_p) = -1$, then $a_p(E) \equiv 0 \pmod{\ell}$.
\end{enumerate}
\end{lemma}
\begin{proof}
\begin{enumerate}
    \item If $p\nmid \ell N_E$, then by Remark \ref{NOS-gal}, $\rho_{E, \ell}$ is unramified at $p$. In particular, $\epsilon_\ell$ is unramified at $p$. When $\ell \nmid N_E$ and $p=\ell$, the claimed property follows by a more delicate analysis; see \cite[p.\ 295, Lemme 2]{Se1972}. 
    \item Let $D$ be the squarefree integer such that $\epsilon_\ell = \chi_D$. It follows from the previous part and (\ref{chiD-ram}) that $\Q(\sqrt{D})/\Q$ is unramified outside of the prime divisors of $N_E$. Thus, $D \mid N_E$.
    \item For $p\nmid N_E$, since $\epsilon_{\ell}(\Frob_p) = -1$, we have that $\bar{\rho}_{E, \ell}(\Frob_p) \in \mathcal{C}_{ns}^+(\ell) \setminus \mathcal{C}_{ns}(\ell)$. From Equation (\ref{NormCartan}), we see that $\tr \gamma = 0$ for each $\gamma \in \mathcal{C}_{ns}^+(\ell) \setminus \mathcal{C}_{ns}(\ell)$. Thus $a_p(E) \equiv \tr(\bar{\rho}_{E, \ell}(\Frob_p)) \equiv 0 \pmod \ell$. \qedhere
\end{enumerate}
\end{proof}

We are now in a position to prove the main theorem.

\begin{proof}[Proof of Theorem \ref{main-thrm}] Let $\ell$ be a prime number such that $\ell > 37$ and  $\bar{\rho}_{E,\ell}$ is non-surjective. By Theorem \ref{SMBPR}, we know that up to conjugation, $G_E(\ell) \subseteq \mathcal{C}_{ns}^+(\ell)$ yet $G_E(\ell) \not\subseteq \mathcal{C}_{ns}(\ell)$. Consider the quadratic Galois character  $\epsilon_\ell$ defined in (\ref{epsilon}). Let $D$ be the squarefree integer such that $\epsilon_\ell = \chi_D$ and consider the quadratic twist $E_D$ of $E$. By Lemma \ref{lemma-epsilon}(\ref{lemma-epsilon-cond}) and (\ref{twist-cond}), we have that
\begin{equation} \rad(N_{E_D}) \quad \text{divides} \quad \rad(2 N_E). \label{pf-main-cond}
\end{equation}
For each prime number $p \nmid 2 N_E$, by (\ref{ap-chi}), we have that
\begin{equation} 
a_{p}(E) = \epsilon_\ell(\Frob_p) a_p(E_D). \label{pf-main-twist}
\end{equation}
By Lemma \ref{non-isogenous-lem}, $E$ and $E_D$ are not $\Q$-isogenous. Thus there exists a prime number $p \nmid 2N_E$ 
such that
\begin{equation} a_{p}(E) \neq a_{p}(E_D).
\label{pf-main-neq}
\end{equation}

Let $p_0$ be the least prime number such that $p_0 \nmid 2N_E$ and the inequality (\ref{pf-main-neq}) holds. Applying Theorem \ref{thrm-qisog} to $E$ and $E_D$ and noting (\ref{pf-main-cond}), we find that
\[p_0\leq (482 \log \rad(2N_{E}) +2880)^2.
\]
Considering (\ref{pf-main-twist}) and (\ref{pf-main-neq}), we see that $\epsilon_\ell(\Frob_{p_0}) = -1$. Thus by Lemma \ref{lemma-epsilon}(\ref{lemma-epsilon-ap}), $\ell$ divides $a_{p_0}(E)$, and so $\ell \leq 
\abs{a_{p_0}(E)}$. The Hasse bound \cite[Theorem V.1.1]{Si2009} gives that $\abs{a_{p_0}(E)} \leq 2 \sqrt{p_0}$. Thus
\[ \ell \leq 2 \sqrt{p_0} \leq 2 (482 \log \rad(2N_{E}) +2880). \]
Expanding the right-hand side, one obtains the claimed bound.
\end{proof}

\subsection{An example}\label{ex-sec} We illustrate Theorem \ref{main-thrm} with a concrete example. Consider the elliptic curve $E/\Q$ with LMFDB \cite{LMFDB} label \href{https://www.lmfdb.org/EllipticCurve/Q/76204800/ut/1}{\texttt{76204800.ut1}}, given by the  Weierstrass equation
\[ y^2=x^3-198450x-27783000. \]
This elliptic curve is without complex multiplication and has conductor
\[ N_E = 76204800 =  2^8 \cdot 3^5 \cdot 5^2 \cdot 7^2. \]
Assuming the generalized Riemann hypothesis, Theorem \ref{main-thrm} tells us that
\[ 
c(E) \leq 964 \log(2\cdot3\cdot5\cdot7) + 5760\approx 10914.61.
\]
In about one second total, SageMath's \cite{SageMath} built-in \texttt{is\_surjective} command confirms that $\bar{\rho}_{E,\ell}$ is surjective for each prime number $\ell \leq 10915$. Thus,  conditional on GRH, $\bar{\rho}_{E,\ell}$ is surjective for all primes $\ell$. Calling Zywina's \texttt{ExceptionalSet} script  \cite{Zy2015b}  on $E$ confirms this unconditionally.

\section{Acknowledgments}

This paper emerged following a series of talks on the effective version of Serre's open image theorem in the Graduate Number Theory Seminar at the University of Illinois Chicago in Spring 2021. Thus, we are thankful to all of the members of the seminar. In addition, we thank Nathan Jones and Sung Min Lee for their helpful comments on an earlier draft of the paper. Further, we thank Jonathan Sorenson for his kind and valuable email response relating to Remark \ref{R:Table}. Lastly, we thank the referee for carefully reading our paper and providing several helpful comments.
\bibliographystyle{amsplain}
\bibliography{Ref}

\end{document}